\documentclass[11pt,leqno]{amsart}
\usepackage{latexsym,amsmath,amssymb}

\title[Density of Lipschitz Mappings]{\protect{Lipschitz Homotopy and Density of Lipschitz Mappings in Sobolev Spaces}}

\author{Piotr Haj\l{}asz}\author{Armin Schikorra}

\address{Piotr Haj\l{}asz, Department of Mathematics, University of Pittsburgh, 301
  Thackeray Hall, Pittsburgh, PA 15260, USA, {\tt hajlasz@pitt.edu}}
\address{Armin Schikorra, Max-Planck Institut MiS Leipzig, Inselstr. 22, 04103 Leipzig, Germany, {\tt armin.schikorra@mis.mpg.de}}

\thanks{P.H. was supported by NSF grant DMS-1161425. A.S. was supported by DAAD fellowship D/12/40670}

\def\eps{\varepsilon}

\def\vi{\varphi}

\def\M{{\mathcal M}}
\def\N{{\mathcal N}}

\newtheorem{theorem}{Theorem}
\newtheorem{lemma}[theorem]{Lemma}
\newtheorem{corollary}[theorem]{Corollary}
\newtheorem{proposition}[theorem]{Proposition}

\def\diam{{\rm diam\,}}

\def\lip{{\rm Lip\,}}

\theoremstyle{definition}
\newtheorem{remark}[theorem]{Remark}
\newtheorem{definition}[theorem]{Definition}

\newtheorem{question}[theorem]{Question}

\newcommand{\barint}{
\rule[.036in]{.12in}{.009in}\kern-.16in \displaystyle\int }

\newcommand{\barcal}{\mbox{$ \rule[.036in]{.11in}{.007in}\kern-.128in\int $}}

\def\eqn#1$$#2$${\begin{equation}\label#1#2\end{equation}}

\def\vint#1_#2{-\kern-#1pt\int_{#2}}

\newcommand{\bbbr}{\mathbb R}

\newcommand{\bbbh}{\mathbb H}

\newcommand{\Sph}{\mathbb S}
\newcommand{\B}{\mathbb B}

\def\mvint_#1{\mathchoice
          {\mathop{\vrule width 6pt height 3 pt depth -2.5pt
                  \kern -8pt \intop}\nolimits_{\kern -3pt #1}}%
          {\mathop{\vrule width 5pt height 3 pt depth -2.6pt
                  \kern -6pt \intop}\nolimits_{#1}}%
          {\mathop{\vrule width 5pt height 3 pt depth -2.6pt
                  \kern -6pt \intop}\nolimits_{#1}}%
          {\mathop{\vrule width 5pt height 3 pt depth -2.6pt
                  \kern -6pt \intop}\nolimits_{#1}}}

\numberwithin{theorem}{section} \numberwithin{equation}{section}

\begin{document}


\sloppy


\begin{abstract}
We construct a smooth compact $n$-dimensional manifold $Y$
with one point singularity such that all its Lipschitz homotopy groups
are trivial, but Lipschitz mappings $\lip(\Sph^n,Y)$ are not dense in
the Sobolev space $W^{1,n}(\Sph^n,Y)$. On the other hand we show that
if a metric space $Y$ is Lipschitz $(n-1)$-connected, then Lipschitz
mappings $\lip(X,Y)$ are dense in $N^{1,p}(X,Y)$ whenever the Nagata dimension
of $X$ is bounded by $n$ and the space $X$ supports the $p$-Poincar\'e inequality. 
\end{abstract}

\maketitle

\section{Introduction}
\label{introduction}

Let $\M$ and $\N$ be compact Riemannian manifolds, $\partial\N=\emptyset$. Being motivated by 
the theory of harmonic mappings Eells and Lemaire, \cite{eellsl},
asked a question whether smooth mappings $C^\infty(\M,\N)$ are dense in the space
of Sobolev mappings between manifolds $W^{1,p}(\M,\N)$.
Here we assume that $\N$ is isometrically embedded in an Euclidean space $\bbbr^\nu$ and we define
$$
W^{1,p}(\M,\N)=\{ u\in W^{1,p}(\M,\bbbr^\nu):\, u(x)\in\N\ \text{a.e.}\}.
$$
The space is equipped with the metric of $W^{1,p}(\M,\bbbr^\nu)$.
Equivalently one may ask about density of Lipschitz mappings. This is indeed equivalent question,
because Lipschitz mappings can be approximated by smooth mappings in the Sobolev norm.
If $p\geq \dim \M$, then smooth mappings are dense in $W^{1,p}(\M,\N)$,
by the theorem of Schoen and Uhlenbeck \cite{schoenu1,schoenu2}, but if $p<\dim\M$, the answer depends
on the topology of manifolds $\M$ and $\N$. 
The following necessary condition for the density is due to
Bethuel and Zheng \cite{bethuelz}.
\begin{proposition}
\label{necessary}
If $\pi_{[p]}(\N)\neq \{0\}$ and $1\leq p<\dim\M$, then the smooth mappings
$C^\infty(\M,\N)$ are not dense in $W^{1,p}(\M,\N)$.
\end{proposition}
Here $\pi_k$ stands for the homotopy group and $[p]$ is the integral part of $p$.
This result is relatively easy to prove (see also \cite{hajlasz2,hajlasz10}).
Bethuel \cite{bethuel1} proved that in the local case (mappings from a ball)
this condition is also sufficient.
The proof of sufficiency is, however, very difficult (see \cite{hangl2} for 
corrections to Bethuel's paper).
\begin{theorem}
\label{bethuel}
If $1\leq p<n$, then smooth mappings $C^\infty(\B^n,\N)$ are dense in
$W^{1,p}(\B^n,\N)$ if and only if $\pi_{[p]}(\N)=0$.
\end{theorem}
In the general case of $W^{1,p}(\M,\N)$, $1\leq p<\dim M$,
the necessary condition $\pi_{[p]}(\N)=0$ is not always sufficient for the density,
see \cite{hangl1} for an example. The necessary and sufficient condition has been discovered 
by Hang and Lin in \cite{hangl2}. While we will not state this condition here, we will
state a sufficient condition that was obtained earlier in 
\cite{hajlasz2}, because this condition will play a role in what is to follow.
\begin{proposition}
\label{h}
If $\pi_1(\N)=\pi_2(\N)=\ldots=\pi_{[p]}(\N)=0$, then
$C^\infty(\M,\N)$ mappings are dense in $W^{1,p}(\M,\N)$.
\end{proposition}
If $p\geq\dim\M$, density is always true, so the interesting case is when 
$1\leq p<\dim\M$.

The theory of Sobolev mappings between manifolds has been extended
to the case of Sobolev mappings with values into metric spaces. The
first papers on this subject include the work of Ambrosio,
\cite{ambrosio}, on limits of classical variational problems and
the work of Gromov and Schoen, \cite{gromovs}, on Sobolev mappings
into the Bruhat-Tits buildings, with applications to rigidity
questions for discrete groups. Later the theory of Sobolev
mappings with values into metric spaces was developed in a more
elaborated form by Korevaar and Schoen, \cite{korevaars}, in their
approach to the theory of harmonic mappings into Alexandrov spaces
of non-positive curvature. Other papers on Sobolev mappings from a
manifold into a metric space include 
\cite{capognal,eellsf,jost1,jost2,jost3,jostz,reshetnyak,serbinowski}. Finally,
analysis on metric spaces, the theory of Carnot--Carath\'eodory
spaces and the theory of quasiconformal mappings between metric
spaces led to the theory of Sobolev mappings between metric
spaces, \cite{heinonenk,HKST,kronz,troyanov},
among which the theory of Newtonian-Sobolev mappings $N^{1,p}(X,Y)$
is particularly important.

The question in what way the results about density of smooth mappings 
between manifolds can be generalized
to the case of mappings between metric spaces was formulated by
Heinonen, Koskela, Shanmugalingam and Tyson, \cite{HKST}.
In this context one asks about density of Lipschitz mappings, because the class of smooth
mappings make no sense. In order to formulate the problem we need to
define the class $N^{1,p}(X,Y)$ of Sobolev mappings between metric spaces.

A {\em metric-measure space} $(X,d,\mu)$ is a metric space $(X,d)$ equipped with a Borel-regular measure $\mu$.
We say that the measure $\mu$ is {\em doubling} if there is a constant $C_d\geq 1$ such that for every ball
$B$ in $X$, $\mu(2B)\leq C_d\mu(B)$. Here and in what follows by $\sigma B$, $\sigma>0$ we denote the ball concentric with 
$B$ and with the radius $\sigma$ times that of $B$. 
In the paper we will {\em always} assume that the measure $\mu$ is doubling.
For simplicity we will also {\em always} assume that the diameter of $X$ is finite. Note that this and the doubling condition 
imply $\mu(X)<\infty$.

Let $U\subset X$ be open. Following \cite{heinonenk} we say that a Borel function $g:U\to [0,\infty]$
is an {\em upper gradient} of a Borel function $u$ defined in $U$ if
$$
|u(\gamma(b))-u(\gamma(a))|\leq \int_\gamma g
$$
for all rectifiable curves $\gamma:[a,b]\to U$. Then we say that the space $(X,d,\mu)$ {\em supports the
$p$-Poincar\'e inequality}, $1\leq p<\infty$ if there are constants $C_P>0$ and $\sigma\geq 1$ such that
for every ball $B$ in $X$, every $u\in L^1(\sigma B)$ and every upper gradient $g$ of $u$ in $\sigma B$
the following version of the Poincar\'e inequality is satisfied
$$
\barint_{B} |u-u_B|\, d\mu \leq C_P (\diam B)\left(\barint_{\sigma B} g^p\, d\mu\right)^{1/p}\, .
$$
Here the barred integral denotes the integral average and $u_B$ is the integral average of $u$ over $B$.

The Sobolev space $N^{1,p}(X,d,\mu)$ was introduced in \cite{shanmugalingam} and it is defined as follows.
We say that $u\in N^{1,p}(X,d,\mu)$ if $u\in L^p$ and there is an upper gradient $g\in L^p$ of $u$. 
The space is equipped with a norm which turns the space into a Banach space:
\begin{equation}
\label{norm}
\Vert u\Vert_{N^{1,p}}=\Vert u\Vert_p+\inf_g\Vert g\Vert_p.
\end{equation}
Here the infimum is taken over all upper gradients of $u$. To be more precise we have to 
identify functions $u,v\in N^{1,p}$ such that $\Vert u-v\Vert_{N^{1,p}}=0$,
just like we identify functions in $L^p$ that are equal almost everywhere. 

If $(V,\Vert\cdot\Vert)$ is a Banach space, then the vector valued Sobolev space $N^{1,p}(X,V)$ 
is defined in a similar manner \cite{HKST}: 
$u\in N^{1,p}(X,V)$ if $u\in L^p(X,V)$ and there is $0\leq g\in L^p(X)$
such that
$$
\Vert u(\gamma(b))-u(\gamma(a))\Vert\leq \int_\gamma g
$$
for all rectifiable curves $\gamma:[a,b]\to X$. Then $N^{1,p}(X,V)$
is a Banach space with respect to the norm \eqref{norm}.
Here we also call $g$ an upper gradient of $u$.
For more details, see \cite{HKST}. 
If $X=\M$ is a Riemannian manifold and $V$ is dual to a separable Banach space,
then this definition is equivalent to the classical definition of the vector
valued Sobolev space $W^{1,p}(\M,V)$ via the distributional derivatives, see
\cite{DHLT,hajlaszt}.

In the case of real valued functions the next result was proved in 
\cite[Theorem~3.2]{hajlaszk} from the Poincar\'e inequality via the telescoping argument.
It was extended to the Banach space valued case in \cite[Proposition~4.6]{HKST}.
Recall that we assumed that $\diam X<\infty$.
\begin{proposition}
\label{pointwise}
Suppose that the space $(X,d,\mu)$ supports the $p$-Poincar\'e inequality,
$(V,\Vert\cdot\Vert)$ is a Banach space, $u\in N^{1,p}(X,V)$, and $g$ is an upper gradient of $u$.
Then the pointwise inequality 
$$
\Vert u(x)-u(y)\Vert \leq Cd(x,y)\left( (\M g^p(x))^{1/p}+ (\M g^p(y))^{1/p}\right)
$$
holds a.e. with some constant $C$ independent of $u$ and $g$. Here $\M$ stands for
the Hardy-Littlewood maximal operator.
\end{proposition}
This inequality implies that on the set $\{ (\M g^p)^{1/p}\leq t\}$ the mapping $u$ is 
$2Ct$-Lipschitz
continuous. Using the Whitney decomposition of the complement of the set 
and the associated Lipschitz partition of unity one can extend the function 
to a Lipschitz function from $X$ to $V$. Since the Lipschitz function
differs from $u$ on a set of small measure standard estimates
lead to the following known result
(see e.g. \cite[Lemma~13]{hajlaszMathAnn}).
\begin{proposition}
\label{lip_app}
Suppose that the space $(X,d,\mu)$ supports the $p$-Poincar\'e inequality for some 
$1\leq p<\infty$ and $V$ is a Banach space. If $u\in N^{1,p}(X,V)$,
then for every $\eps>0$ there is $u_\eps\in \lip(X,V)$ such that
$\mu(\{x:\, u(x)\neq u_\eps(x)\})<\eps$ and $\Vert u-u_\eps\Vert_{1,p}<\eps$.
\end{proposition}
Every separable metric space admits an isometric embedding into the Banach space of bounded
sequences $\ell^\infty$ (the Kuratowski embedding). This can be used to define the space of Sobolev mappings 
$N^{1,p}(X,Y)$ with values in $Y$. Namely if $\lambda:Y\to V$ is an isometric embedding 
then we define
$$
N^{1,p}(X,Y)=\{ u:X\to Y:\, \lambda\circ u\in N^{1,p}(X,V)\}.
$$
The space is equipped with the metric $d(u,v)=\Vert \lambda\circ u-\lambda\circ v\Vert_{N^{1,p}}$.
This definition resembles the definition of the class of Sobolev mappings between smooth manifolds,
but now instead of taking an embedding of the manifold into a Euclidean space we take 
an embedding of a metric space into a Banach space.

If $X=\M$ is a Riemannian manifold and $Y$ is embedded in a Banach space $V$ that is dual to 
a separable Banach space (that is always possible, because we can take $V=\ell^\infty$), then
we will write $W^{1,p}(\M,Y)$, because, as it was explained earlier, we have $N^{1,p}(\M,V)=W^{1,p}(\M,V)$.

Observe that each mapping $u\in N^{1,p}(X,Y)$ (or rather $\lambda\circ u$)
can be approximated by Lipschitz mappings with values
into $V$ (Proposition~\ref{lip_app}). In this setting Heinonen, Koskela, Shanmugalingam and Tyson \cite[Remark~6.9]{HKST} asked: 
{\em It is an interesting problem to determine
when one can choose the Lipschitz approximation to have values in the target $Y$. [\ldots]
For instance, one can ask to what extent Bethuel's results have analogs for general targets.}
The following partial answer was obtained in \cite[Theorem~6]{hajlaszMathAnn}.
\begin{theorem}
\label{h2}
Let $Y$ be a finite Lipschitz polyhedron and $1\leq p<\infty$. Then the class of Lipschitz
mappings $\lip(X,Y)$ is dense in $N^{1,p}(X,Y)$ for every metric-measure space $X$
of finite measure that supports the $p$-Poincar\'e inequality if and only if 
$\pi_1(Y)=\pi_2(Y)=\ldots=\pi_{[p]}(Y)=0$.
\end{theorem}
Note that this condition appeared also in Proposition~\ref{h}, but now it is necessary and sufficient.
While the theorem treats a general class of spaces $X$ as a source, the targets are
Euclidean-like, and it is an interesting question to investigate more general
targets. Of particular interest is the Heisenberg group $\bbbh_n$ which is a fundamental
example of a metric space in the analysis on metric spaces. The motivation for the investigation
of density of Lipschitz mappings in $N^{1,p}(X,\bbbh_n)$ stems from the theory of
harmonic mappings with values in the Heisenberg group developed
by Capogna and Lin \cite{capognal}, just like the theory of harmonic mappings
was the original motivation for the questions of Eells and Lemaire. 
This problem  was investigated in \cite{DHLT,hajlaszst}. 

In the case of mappings between manifolds, the answer to the density question 
depends on the topology of manifolds and in particular on the homotopy groups
of the target, see Proposition~\ref{h} and Theorem~\ref{h2}. 
The Heisenberg group $\bbbh_n$ is homeomorphic
to the Euclidean space $\bbbr^{2n+1}$ and thus its homotopy groups are trivial. 
It turns out, however, that more appropriate objects to consider in the case of metric
targets are the {\em Lipschitz homotopy groups}. For the following definition, see \cite{DHLT}.

\begin{definition}
Let $(Y, y_0)$ be a pointed metric space. We define the {\em Lipschitz homotopy group}
$\pi_n^{\lip}(Y,y_0)$ in the same way as the classical homotopy group \cite{hatcher},
with the exception that both mappings and homotopies are required to be Lipschitz.
We {\em do not} require that the Lipschitz constant of the homotopy
between Lipschitz mappings $f,g:(Q^n,\partial Q^n)\to (Y,y_0)$, where $Q^n=[0,1]^n$, 
is comparable to the larger of the Lipschitz constant of the mappings $f,g$.
\end{definition}
In particular
$\pi_0^\lip(Y,y_0)$ is the set of Lipschitz-path-connected components, i.e. components
in which any two points can be connected by a rectifiable curve, and $\pi_0^\lip(Y,y_0)=0$
means that the space is rectifiably connected.
The following result is easy to prove \cite{DHLT}. 
\begin{proposition}
$\pi_n^\lip(Y,y_0)=0$ if and only if
every Lipschitz map $(\Sph^n,s_0)\to (Y,y_0)$ admits a Lipschitz extension
$\B^{n+1}\to Y$.
\end{proposition}
Here and in what follows $\Sph^n$ stands for the unit sphere $\Sph^n(0,1)$
in $\bbbr^{n+1}$.

In the case of compact smooth manifolds Lipschitz homotopy groups are equivalent to the classical
homotopy groups, because continuous mappings and homotopies can be approximated by smooth ones,
however, for non-smooth spaces the situation is different.
In \cite{DHLT} it was proved that $\pi_n^\lip(\bbbh_n)\neq 0$
and it was used to prove the corresponding lack of density of Lipschitz mappings:
{\em If $\M$ is a compact Riemannian manifold of dimension $\dim\M\geq n+1$,
then Lipschitz mappings $\lip(\M,\bbbh_n)$ are not dense in $W^{1,p}(\M,\bbbh_n)$
when $n\leq p<n+1$}. Similarly using a generalized Hopf invariant 
it was proved in \cite{hajlaszst} that $\pi^{\lip}_{4n-1}(\bbbh_{2n})\neq 0$
and the following result was concluded from it:
{\em If $\M$ is a compact Riemannian manifold of dimension $\dim\M\geq 4n$,
then Lipschitz mappings $\lip(\M,\bbbh_{2n})$ are not dense in 
$W^{1,p}(\M,\bbbh_{2n})$ when $4n-1\leq p<4n$.}
Thus the two non-density results are counterparts of Proposition~\ref{necessary}
for the Heisenberg group targets, so one might expect that this result 
extends for other metric targets. Thus one might ask the following question
for the class of mappings from a manifold $\M$ to a metric space $Y$:
\begin{question}
\label{Q1}
Suppose that $\pi_{[p]}^\lip(Y)\neq 0$ and $1\leq p<\dim\M$.
Is it true that Lipschitz mappings $\lip(\M,Y)$ are not dense in 
$N^{1,p}(\M,Y)$?
\end{question}
We do not know the answer. One may ask also whether Proposition~\ref{h}
and Theorem~\ref{h2} extend to more general targets.
\begin{question}
Suppose that $\pi_1^\lip(Y)=\ldots=\pi_{[p]}^\lip(Y)=0$.
Is it true that Lipschitz mappings $\lip(\M,Y)$ are dense in 
$W^{1,p}(\M,Y)$?
\end{question}
It turns out that the
answer is in the negative as the following results shows.
\begin{theorem}
\label{hs1}
For $n\geq 2$ there is a compact subset $Y\subset\bbbr^{n+1}$
and a point $p\in Y\cap\Sph^n$ such that
\begin{itemize}
\item $Y$ is homeomorphic to $\Sph^n$;
\item $Y\setminus \{ p\}$ is a smooth manifold diffeomorphic to 
$\Sph^n\setminus \{ p\}$. In fact there is a Lipschitz
continuous homeomorphism $\Phi:Y\to \Sph^n$ such that 
$\Phi:Y\setminus \{ p\}\to\Sph^n\setminus \{ p\}$ is a smooth diffeomorphism;
\item $\pi^\lip_k(Y)=0$ for all $k\geq 1$ (for any choice of $y_0$);
\item Lipschitz mappings $\lip(\Sph^n,Y)$ are not dense in $W^{1,n}(\Sph^n,Y)$.
\end{itemize}
\end{theorem}
Although $Y$ is Lipschitz homeomorphic to $\Sph^n$, it is not bi-Lipschitz
homeomorphic to the sphere.
The space $Y$ is not rectifiably connected: any curve in $Y$ connecting 
$p$ to another point has infinite length. Thus $\pi_0(Y)\neq 0$.
By adding a segment to $Y$ connecting $p$ with the antipodal point of $\Sph^n$, 
we obtain a set $Z$ that is rectifiably connected and has similar 
properties as $Y$.
\begin{theorem}
\label{hs3}
For $n\geq 2$ there is a compact set $Z\subset\bbbr^{n+1}$ 
such that $\pi_k^\lip(Z)=0$ for all $k\geq 0$ and 
Lipschitz mappings $\lip(\Sph^n,Z)$ are not dense in $W^{1,n}(\Sph^n,Z)$.
\end{theorem}
Thus in order to obtain positive density results we need a stronger condition than
vanishing of Lipschitz homotopy groups.
The next definition is taken from \cite{langs}.
\begin{definition}
A metric space $Y$ is {\em Lipschitz $n$-connected} for some integer
$n\geq 0$ if there is a constant $\gamma\geq 1$ such that for each 
$k\in \{0,1,\ldots,n\}$, every $L$-Lipschitz map 
$f:\Sph^k\to Y$ admits a $\gamma L$-Lipschitz extension 
$F:\B^{k+1}\to Y$.
\end{definition}
The condition that the space is Lipschitz $n$-connected is stronger
than vanishing of the Lipschitz homotopy groups for $k\leq n$, because
we want to control the Lipschitz constant of the extension. 
As it will be explained later the set $Z$ is not Lipschitz
($n-1$)-connected despite the fact that all its Lipschitz homotopy groups are trivial.

\begin{definition}
The {\em Nagata} dimension $\dim_N X$ of a metric space $X$ is the least integer
$n$ with the property that there is $C>0$ such that for any $s>0$,
there is a covering $X=\bigcup_{i\in I} A_i$ such that
\begin{itemize}
\item $\diam A_i\leq Cs$ for all $i\in I$;
\item Every ball $B(x,s)$ intersects at most $n+1$ sets $A_i$.
\end{itemize}
If such $n$ does not exist, then $\dim_N X=+\infty$.
\end{definition}
The following result is easy to prove (see e.g. \cite[Lemma~2.3]{langs})
\begin{proposition}
If $X$ is equipped with a doubling measure, then $\dim_N X<\infty$.
\end{proposition}
For spaces that are $n$-Lipschitz connected we have the following density result
which in some sense is a counterpart of Proposition~\ref{h} and
Theorem~\ref{h2} but with an
important additional restriction on the dimension of the domain,
see also Question~\ref{Q1}.
\begin{theorem}
\label{hs2}
Suppose that the space $(X,d,\mu)$ supports
the $p$-Poincar\'e inequality and $\dim_N X\leq n$.
If a separable metric space $Y$ is Lipschitz $(n-1)$-connected, then 
Lipschitz mappings $\lip(X,Y)$ are dense in $N^{1,p}(X,Y)$.
\end{theorem}
\begin{remark}
If a space supports the $p$-Poincar\'e inequality, then it supports the $q$-Poincar\'e
inequality for all $p\leq q<\infty$ (by H\"older's inequality) and hence we have density in 
$N^{1,q}$ for all $q\geq p$. If in addition $X$ is complete and $p>1$, then 
there is $\eps>0$ such that the space supports the $q$-Poincar\'e inequality for all 
$p-\eps<q<\infty$, \cite{keithz}, and thus the density is true for this range of $q$.
\end{remark}
\begin{remark}
If $X$ supports $1$-Poincar\'e inequality, then we have density for all $1\leq p<\infty$.
This is consistent with results for mappings between manifolds: if
$\pi_1(\N)=\ldots=\pi_{n-1}(\N)=0$, and $\dim\M\leq n$,
then Proposition~\ref{h} gives density for $1\leq p<n$,
but if $p\geq n$ we always have density by the result of Schoen-Uhlenbeck \cite{schoenu1,schoenu2}.
\end{remark}
\begin{remark}
Theorems~\ref{hs2} and~\ref{hs3} show that the set $Z$ is not Lipschitz ($n-1$)-connected,
but it can be checked more directly.
Every Lipschitz mapping $f:\Sph^k\to Z$, $k\geq 0$
admits a Lipschitz extension $F:\B^{k+1}\to Z$, but 
using the construction of $Z$ one can show that for $k=0$ and $k=n-1$,
there is no
constant $C\geq 1$ with the property that every $L$-Lipschitz mapping
$f:\Sph^{k}\to Z$ admits a $CL$-Lipschitz extension $F:\B^{k+1}\to Z$.
This is a good exercise for the reader. We do not provide details here, because 
this direct argument will play no role in the paper.
\end{remark}
\begin{remark}
According to \cite[Proposition~2.13]{heinonen}
Lipschitz $n$-connected sets in $\bbbr^{n+1}$ are Lipschitz 
retracts of $\bbbr^{n+1}$. This also follows from a more difficult
result of Lang and Schlichenmaier, Lemma~\ref{us} below. 
From this fact the density of Lipschitz mappings
in $W^{1,p}$ for all $1\leq p<\infty$ easily follows, 
see also \cite[Theorem~1.3]{hajlaszGAFA}.
\end{remark}

It was proved in \cite{wengery} that the Heisenberg group $\bbbh_n$ is Lipschitz
$(n-1)$-connected and hence we obtain
\begin{corollary}
Suppose that the space $(X,d,\mu)$ supports
the $p$-Poincar\'e inequality and $\dim_N X\leq n$.
Then Lipschitz mappings $\lip(X,\bbbh_n)$ are dense in 
$N^{1,p}(X,\bbbh_n)$.
\end{corollary}
This result generalizes Theorem~1.2(b) from \cite{DHLT}, where the case of $X$ being 
a compact manifold was considered.

Notation is pretty standard. By $C$ we will denote a general constant whose value
may change in a single string of estimates. The paper is organized as follows. 
In Section~\ref{proof} we prove Theorem~\ref{hs2} and in Section~\ref{example}
we prove Theorems~\ref{hs1} and~\ref{hs3}.

\section{Proof of Theorem~\ref{hs2}}
\label{proof}

The following result due to Lang and Schlichenmaier \cite[Theorem~1.5]{langs}
will play a fundamental role in the proof.
\begin{lemma}
\label{us}
Suppose that $X$ and $Y$ are metric spaces such that $\dim_N X\leq n$ and 
$Y$ is Lipschitz $(n-1)$-connected. Then there is a constant $C\geq 1$ such that
for any closed set $Z\subset X$ and any $L$-Lipschitz map $f:Z\to Y$
there is a $CL$-Lipschitz extension $F:X\to Y$.
\end{lemma}
Since the metric in $N^{1,p}(X,Y)$ is defined via an isometric embedding of $Y$ into a
Banach space, we may assume that $Y\subset V$ is a subset of a Banach space $V$. 
Let $u\in N^{1,p}(X,Y)\subset N^{1,p}(X,V)$. 
Let $0\leq g\in L^p(X)$ be an upper gradient of $u$.
According to Proposition~\ref{pointwise}
$$
\Vert u(x)-u(y)\Vert_V \leq Cd(x,y)\left( (\M g^p(x))^{1/p}+(\M g^p(y))^{1/p}\right).
$$
Let
$$
E_t = \{x\in X: (\M g^p(x))^{1/p}\leq t \}.
$$
It follows from the weak type estimates of the maximal function that $t^p\mu(X\setminus E_t)\to 0$
as $t\to \infty$. Hence we can find a closed set $F_t\subset E_t$ such that $t^p\mu(X\setminus F_t)\to 0$.
The function $u$ restricted to $F_t$ is $2Ct$-Lipschitz continuous. 
According to Lemma~\ref{us} there is a $C't$-Lipschitz
extension $u_t:X\to Y$. It remains to prove that $\Vert u-u_t\Vert_{N^{1,p}}\to 0$ as $t\to\infty$.
Since $X$ is bounded and the mapping $u_t$ is $C't$-Lipschitz, we conclude that 
$\Vert u_t\Vert_V\leq C+Ct\diam X\leq C'(1+t)$. Thus
\begin{eqnarray*}
\int_X \Vert u-u_t\Vert_V^p 
& = &\int_{X\setminus F_t} \Vert u-u_t\Vert_V^p \\
& \leq & C\int_{X\setminus F_t} \Vert u\Vert_V^p + C(1+t)^p\mu(\{ X\setminus F_t\}) \to 0
\end{eqnarray*}
as $t\to \infty$. The following elementary lemma shows a localization
property of upper gradients;
the proof is quite standard and left to the reader.
\begin{lemma}
If $0\leq g\in L^p(X)$ is an upper gradient of $f\in N^{1,p}(X,V)$ and $f$ is constant on a closed set $E\subset X$,
then $h=g\chi_{X\setminus E}$ is an upper gradient of $f$.
\end{lemma}
Since $u_t$ is $Ct$-Lipschitz, the constant function $Ct$ is its upper gradient and since $u-u_t=0$
on the closed set $F_t$ it follows from the lemma that $h=(g+Ct)\chi_{X\setminus F_t}$ is an upper
gradient of $u-u_t$. We have
$$
\int_X h^p\, d\mu \leq \int_{X\setminus F_t} (g+Ct)^p\, d\mu \to 0
$$
as $t\to\infty$. Thus $\Vert u-u_t\Vert_{N^{1,p}}\to 0$ as $t\to\infty$.
The proof is complete.
\hfill $\Box$

\section{Example}
\label{example}

In this section we will prove Theorems~\ref{hs1} and~\ref{hs3}. In the first subsection we will provide
details of the construction of the set $Y$ and we will prove all its properties except the lack of
density of Lipschitz mappings.
In the second subsection we will construct set $Z$ and we will prove that $\pi_k^\lip(Z)=0$ for all $k\geq 0$.
In the last subsection we will prove the lack of density of Lipschitz mappings in
$W^{1,n}(\Sph^n,Y)$ and $W^{1,n}(\Sph^n,Z)$.

\subsection{Construction of $Y$}
The main idea is to pick a point $p \in \Sph^n$ and add continuous oscillations into the normal direction around that point. 
The oscillations should be so that the resulting set is still homeomorphic to the sphere with the 
homeomorphism $\phi$ being in $W^{1,n}$.

The amplitude of the oscillations is described by $a \in C^0([0,\infty))\cap C^\infty((0,\infty))$
$$
a(t) := \begin{cases}
          0 \quad & t = 0,\\
          \displaystyle{\frac{\sin(\log(\log({e}/{t})))}{1+\log(\log({e}/{t}))}} \quad &t \in (0,1],\\
          \mbox{smooth extension} \quad &t \in (1,3/2),\\
          0 \quad & t \geq 3/2.
         \end{cases}
$$
Fix a point $p\in\Sph^n$ and define
$\phi:\Sph^n\to\bbbr^{n+1}$ by
$$
\phi(x)=\left(1+\frac{1}{2}a(|x-p|)\right)x.
$$
The set $Y$ is defined as the image of the mapping $\phi$, $Y=\phi(\Sph^n)$.
Since $\phi \in C^0(\Sph^n,\bbbr^{n+1})$ translates points of the sphere along the direction 
of the outer normal to the sphere, $\phi$ has no self-intersections and thus it is 
a homeomorphism. Clearly $Y\setminus \{ p\}$ is a smooth manifold.
Let $\Phi=\phi^{-1}:Y\to\Sph^n$ be the inverse homeomorphism. 
Obviously, $\Phi$ is the restriction of the projection mapping
$$
\pi(x)=\frac{x}{|x|}\in C^\infty(\bbbr^{n+1}\setminus \{ 0\},\Sph^n)
$$
to $Y$. Hence $\Phi$ is Lipschitz continuous and its restriction to $Y\setminus \{ p\}$
is a smooth diffeomorphism. 
The function $a$ was chosen in order to satisfy the claim of the next lemma.
\begin{lemma}
\label{pr:fbehviour}
\begin{equation}
\label{p1}
\int_{0}^s |a'(t)|\ dt = \infty \quad \mbox{for any $s > 0$},
\end{equation}
but
\begin{equation}
\label{p2}
 \int_{0}^1 |a'(t)|^n t^{n-1}\ dt < \infty \quad \mbox{for all $n \geq 2$.}
\end{equation}
\end{lemma}
\begin{proof}
Since
$$
(\log(\log({e}/{t})))' = -\frac{1}{t\log({e}/{t})}\, ,
$$
for $0<t<1$ we have
$$
a'(t)
 = 
\frac{-\cos(\log(\log({e}/{t})))}{t\log({e}/{t})\cdot (1+\log(\log({e}/{t})))}
 + 
\frac{\sin(\log(\log({e}/{t})))}{t\log({e}/{t})\cdot (1+\log(\log({e}/{t})))^2}\,.
$$
Thus for any $0<s<1$
the change of variables $x(t)=\log(\log(e/t))$ yields
\begin{eqnarray*}
\int_0^s |a'(t)|\, dt 
& = &
\int_{x(s)}^\infty \left| \frac{-\cos x}{x+1} + \frac{\sin x}{(x+1)^2}\right|\, dx \\
& \geq &
\int_{x(s)}^\infty \frac{|\cos x|}{x+1}\, dx - \int_{x(s)}^\infty \frac{dx}{(x+1)^2} = \infty.
\end{eqnarray*}
On the other hand for $n\geq 2$ we have
$$
\int_0^1 |a'(t)|^n t^{n-1}\, dt \leq 2^n\int_0^1 \frac{dt}{t\log^n(e/t)} =2^n\int_1^\infty \frac{dx}{x^n} <\infty.
$$ 
The proof is complete.
\end{proof}
\begin{lemma}
\label{sob_hom}
$\phi\in W^{1,n}(\Sph^n,Y)$.
\end{lemma}
\begin{proof}
It suffices to show that the derivative of $\phi$ is integrable with the exponent $n$
in some neighborhood of the singularity $p$. Using a suitable coordinate system in a neighborhood of $p\in\Sph^n$
both in the domain and in the target $\bbbr^{n+1}$, the image of $\phi$ becomes the graph of the function $a(|x|)$.
The fact that the function $a(|x|)$ has derivative integrable with exponent $n$ follows from \eqref{p2}
after integration in the polar coordinate system:
$$
\int_{\B^n(0,1)} |\nabla (a(|x|))|^n\, dx =C\int_0^1 |a'(t)|^nt^{n-1}\, dt<\infty.
$$  
The proof is complete.
\end{proof}

In the last subsection we will prove that the mapping $\phi\in W^{1,n}(\Sph^n,Y)$ cannot be approximated by Lipschitz 
maps $\lip(\Sph^n,Y)$ in the Sobolev norm. 
But now we will prove that $\pi^\lip_k(Y)=0$ for all $k\geq 1$. This immediately follows from
the next lemma.
\begin{lemma}
\label{pi_lip}
If $k\geq 1$ and $f:\Sph^k\to Y$ is Lipschitz continuous, then 
$f(\Sph^k)=\{ p\}$ or $p\not\in f(\Sph^k)$.
\end{lemma}
Indeed, if $p\not\in f(\Sph^k)$, then the image of $\phi$ omits a certain neighborhood of $p$
and hence it is contained in a subset of $Y$ which is diffeomorphic to $\B^n$ and thus Lipschitz
contractible.

\begin{proof}[Proof of Lemma~\ref{pi_lip}]
Since $f$ is Lipschitz,
any two points in the image of $f(\Sph^n)$ can be connected by a rectifiable curve.
Thus it suffices to show that the point $p\in Y$ cannot be connected to any other point in $q\in Y$
by a rectifiable curve.
Using the coordinate system near $p$ as in the previous proof we can represent a neighborhood of
$p\in Y$ as the graph of the function $a(|x|)$ defined in $\B^n$. In this coordinate system $p=0$. 
Consider the plane through $p$ and $q$ that is orthogonal to $\B^n$. The curve obtained in the 
intersection of this plane with the graph of $a(|x|)$ is the graph of the function $a(t)$ of one variable.
This curve connects $p$ to $q$. If the Euclidean distance of the projections of
$p$ and $q$ onto $\B^n$ equals $s$, the length of this curve is
$$
\int_0^s \sqrt{1+a'(t)^2}\, dt = \infty
$$
by \eqref{p1}. Note that no other curve connecting $p$ to $q$ can be shorter.
The proof of the lemma is complete. This also completes the proof of the fact that
$\pi_k^\lip(Y)=0$ for $k\geq 1$.
\end{proof}

\subsection{Construction of $Z$}
In this section we will provide the construction of the set $Z$ from Theorem~\ref{hs3} and we will prove that
$\pi_k^\lip(Z)=0$ for all $k\geq 0$. The proof of the lack of density will be given in the next section.
The set $Z$ is simply obtained from $Y$ by adding the segment $I$ that connects the point $p$ through the center of the
sphere with the antipodal point of $\Sph^n$. Note that this antipodal point belongs also to $Y$.
Clearly the set $Z$ is now rectifiably connected since the point 
$p$ can be connected with the rest of the set through the added segment, so $\pi_0^\lip(Z)=0$.
The fact that $\pi_k^\lip(Z)$ for $k\geq 1$ is also easy. If $f:\Sph^k\to Z$ is Lipschitz, then 
$f(\Sph^k)$ is contained in $Y$ with a neighborhood of $p$ removed, plus $I$. This set is Lipschitz contractible.
\hfill $\Box$

\subsection{Proof of the lack of density of Lipschitz maps}
In this section we will complete the proofs of Theorems~\ref{hs1} and~\ref{hs3} by showing 
the lack of density of Lipschitz mappings.

\begin{proof}[Proof of Theorem~\ref{hs1}]
We will prove that the mapping $\phi\in W^{1,n}(\Sph^n,Y)$ cannot be approximated by Lipschitz mappings
$\lip(\Sph^n,Y)$ in the $W^{1,n}$ norm. Suppose by contrary that there is a sequence
$g_k\in \lip(\Sph^n,Y)$ such that 
\begin{equation}
\label{123}
\Vert\phi-g_k\Vert_{1,n}\to 0
\quad
\mbox{as $k\to\infty$.}
\end{equation}
By Lemma~\ref{pi_lip},
$p\not\in g_k(\Sph^n)$. As we know $\pi=\Phi=\vi^{-1}:Y\to \Sph^n$ is 
smooth in a neighborhood of $Y$ in $\bbbr^{n+1}$. 
Hence $f_k:=\Phi\circ g_k\in \lip(\Sph^n,\Sph^n\setminus\{ p\})$,
$\Phi\circ\phi={\rm id}\, :\Sph^n\to\Sph^n$ and \eqref{123} yields 
\begin{equation}
\label{234}
\Vert f_k-{\rm id}\Vert_{1,n}\to 0
\quad
\mbox{as $k\to\infty$.}
\end{equation}
Degree of a Lipschitz mapping can be expressed as the integral of
the Jacobian. The integral of the Jacobian is continuous in the $W^{1,n}$ norm. That easily follows
from the definition of the Jacobian and H\"older's inequality. Hence \eqref{234} implies
that the degree of the mapping $f_k$ converges to the degree of the identity mapping, i.e.
it converges to $1$, but this is impossible, because the mapping $f_k$ is not surjective and hence its degree 
equals zero.
The proof is complete.
\end{proof}

\begin{proof}[Proof of Theorem~\ref{hs3}]
Now we will prove that Lipschitz mappings $\lip(\Sph^n,Z)$ are not dense in $W^{1,n}(\Sph^n,Z)$.
Namely we will prove that $\phi\in W^{1,n}(\Sph^n,Y)\subset W^{1,n}(\Sph^n,Z)$ cannot be 
approximated by Lipschitz mappings $\lip(\Sph^n,Z)$. By contrary suppose that
$g_k\in \lip(\Sph^n,Z)$ is a Lipschitz approximation. 
Denote the endpoints of the segment $I$ by $p$ and $q$, i.e. $q$ is the antipodal point to $p$.
The image of $g_k$ is contained in $Y$ with a neighborhood of $p$ removed, plus $I$.
Thus composing $g_k$  with the retraction of $I$ onto $q$
gives a Lipschitz map $\tilde{g}_k:\Sph^n\to Y$. 
Since $g_k\to\phi$ in $L^n$, it easily follows that the measure of the set $g_k^{-1}(I)$
converges to zero as $k\to\infty$. Since $\tilde{g}_k$ differs from $g_k$ only on the set
$g_k^{-1}(I)$ one can easily show that $\Vert g_k-\tilde{g}_k\Vert_{1,n}\to 0$ and hence
$\Vert \tilde{g}_k-\phi\Vert_{1,n}\to 0$.
Now the result follows directly from the previous proof, because Lipschitz mappings $\tilde{g}_k$
into $Y$ cannot approximate $\phi$ in the $W^{1,n}$ norm.
\end{proof}

\end{document}